\documentclass[12pt,twoside]{amsart}
\usepackage{amsfonts}
\usepackage{amsmath}
\usepackage{amssymb}
\usepackage{graphicx}
\usepackage{mathrsfs}
\usepackage{multicol}
\usepackage{color}
\usepackage{pst-eucl,pstricks-add}
\usepackage{pst-3dplot}
\usepackage{pst-solides3d}

\usepackage{lscape}

\newtheorem{theo}{Theorem}

\newtheorem{lem}{Lemma}
\newtheorem{prop}{Proposition}
\theoremstyle{definition}
\newtheorem{defn}{Definition}
\theoremstyle{remark}
\newtheorem{rem}{\bf Remark\/}

\numberwithin{equation}{section}

\definecolor{bronze}{rgb}{0.4, 0.7, 0.1}

\def\1{{\mathchoice {\rm 1\mskip-4mu l} {\rm 1\mskip-4mu l}{\rm 1\mskip-4.5mu l} {\rm 1\mskip-5mu l}}}
\newcommand{\ds}{\displaystyle}

\usepackage[width=16.00cm, height=21.00cm, left=2.00cm, right=2.00cm, top=4cm, bottom=2cm]{geometry}

\title{Berezin  transforms on modified Bergman spaces}

\author[N. Ghiloufi and S. Snoun]{Noureddine Ghiloufi and Safa Snoun}
\email{noureddine.ghiloufi@fsg.rnu.tn, snoun.safa@gmail.com}
\address{University of Gabes\\ Faculty of Sciences of Gabes\\ LR17ES11 Mathematics and Applications\\ 6072, Gabes, Tunisia.}

\begin{document}

\begin{abstract}
    In this paper we study the continuity of the Berezin transform on modified Bergman spaces and we establish a Lipschitz estimate in terms of the Bergman-Poincar\'e metric.
\end{abstract}
\maketitle
\section{Introduction}
The Berezin transform $\mathbb B$ was introduced first by Berezin in \cite{Be} to solve some physical problems. This transformation has many applications in physics and mathematics, namely in complex analysis, it characterizes some functional spaces and the typical example is the so-called Besov spaces. Moreover, it plays a fundamental role in the study of many other operators (say Toeplitz and Hankel for example). For harmonic functions, it is necessary to replace the Poisson transformation with the Berezin transform in the context of classical Bergman spaces. Indeed, a function in the Bergman space will be harmonic if and only if it is a fixed point for $\mathbb B$.  This result can be deduced from the fact that $\mathbb B$ commutes with the Laplace operator $\Delta$ and every automorphism of the unit disk. \\
In recent years, the classical Bergman spaces were extended to what we call modified Bergman spaces so it is logical to study the corresponding Berezin transform on these new spaces. In fact, this will be the aim of this paper. The first problem that we face in this case is that the measures considered (and consequently the Berezin transform) are not invariant by the set of automorphisms of the disk; so the situation will be more difficult and techniques of proofs must be adopted for the present case. \\

Other than the introduction, the paper contains three sections, in the first one we give some preliminary results about the reproducing kernel  $\mathbb K_{\alpha,\beta}$ of the modified Bergman space $\mathcal A^2(\mathbb D,\mu_{\alpha,\beta})$. Then, in the second section, we study properties of the Berezin transform $\mathbb B_{\alpha,\beta}$. Namely, we prove the first main result of the paper:
\begin{theo}\label{th1}
Let $-1<\alpha,a,b<+\infty,\ -1<\beta\leq0$ and $1\leq p<+\infty$. Then the Berezin transform $\mathbb B_{\alpha,\beta}:L^p(\mathbb D,\mu_{a,b})\longrightarrow L^p(\mathbb D,\mu_{a,b})$ is bounded if and only if
$$p(\alpha+1)>a+1\quad and\quad \left\{
\begin{array}{lcl}
   b\leq \beta& if& p=1\\
   b< p(\beta+1)-1& if& p>1.
\end{array}
\right.$$
\end{theo}
To improve a result due to B\'ekoll\'e, Berger, Coburn and Zhu (see \cite{BBCZ} or \cite{He-Ko-Zh}), we recall in the last section,  the Bergman-Poincar\'e metric introduced by the second author in \cite{Sn} and we use the Berezin transform $\mathbb B_{\alpha,\beta}$ to define the space BMO and finally we prove the second main result of the paper:
\begin{theo}\label{th2}
    For every $f\in BMO$, we have
    $$|\mathbb B_{\alpha,\beta}f(z)-\mathbb B_{\alpha,\beta}f(w)|\leq 2\|f\|_{BMO}\mathbf{d}_{\alpha,\beta}(z,w)
    $$
    for all $z,w\in\mathbb D$.
\end{theo}

\section{Preliminaries}
In \cite{Gh-Za} the authors introduced, for every $\alpha,\beta>-1$, the modified Bergman space  $\mathcal A^p(\mathbb D,\mu_{\alpha,\beta})$, of holomorphic functions on the punctured unit disk $\mathbb D^*$  that are $p-$integrable with respect to the probability measure $\mu_{\alpha,\beta}$ defined by  $$d\mu_{\alpha,\beta}(z):=\frac{|z|^{2\beta}(1-|z|^2)^\alpha }{\mathscr B(\alpha+1,\beta+1)}dA(z)$$ where $\mathscr B$ is the beta function given by $$\mathscr B(a,b)=\int_0^1x^{a-1}(1-x)^{b-1}dx=\frac{\Gamma(a)\Gamma(b)}{\Gamma(a+b)},\quad \forall\; a,b>0$$ and $$dA(z)=\frac{1}{\pi}dx dy=\frac{1}{\pi}rdrd\theta,\quad z=x+iy=re^{i\theta}.$$
In the particular case $p=2$, it is shown that $\left(\mathcal A^2(\mathbb D,\mu_{\alpha,\beta}),\langle .,.\rangle_{\mu_{\alpha,\beta}}\right)$ is a Hilbert space where the inner product $\langle .,.\rangle_{\mu_{\alpha,\beta}}$ is inherited from $L^2(\mathbb D,\mu_{\alpha,\beta})$:
$$\langle f,g\rangle_{\mu_{\alpha,\beta}}:=\int_{\mathbb D}f(z)\overline{g(z)}d\mu_{\alpha,\beta}(z),\quad \forall\; f,g\in L^2(\mathbb D,\mu_{\alpha,\beta}).$$
The reproducing kernel of the modified Bergman space $\mathcal A^2(\mathbb D,\mu_{\alpha,\beta})$ was determined later in \cite{Gh-Sn, AG-Sn} as $\mathbb K_{\alpha,\beta}(z,w)=\mathcal K_{\alpha,\beta}(z\overline{w})$ where
\begin{equation}\label{eq(2.1)}
\mathcal K_{\alpha,\beta}(\xi)=\ds\frac{(\beta_0+1)_s}{(\alpha+\beta_0+2)_s}\frac{1}{\xi^s}\ _2F_1\left(\left.
\begin{array}{cc}
1,&\alpha+\beta_0+2\\
& \beta_0+1
\end{array}\right|\xi\right)
\end{equation}

where $s=\inf\{k\in\mathbb N;\ \beta\leq k\}$ and $\beta=\beta_0+s$ (so  $\beta_0\in]-1,0]$).\\
Here $\ _2F_1$ is the Gauss hypergeometric function. Indeed, many results will be expressed in terms of the hypergeometric functions:
    $$_kF_j\left(\left.
\begin{array}{c}
a_1,\dots,a_k\\
b_1,\dots,b_j
\end{array}\right|\xi\right)=\sum_{n=0}^{+\infty}\frac{(a_1)_n\dots(a_k)_n}{(b_1)_n\dots(b_j)_n}\frac{\xi^n}{n!}$$
where $(a)_n=a(a+1)\dots(a+n-1)$ is the Pochhammer symbol for every $a\in\mathbb R$. In fact  $(a)_n=\Gamma(a+n)/\Gamma(a)$  for every $a\not\in\mathbb Z^-$. (see \cite{Ma-Ob-So} for more details about hypergeometric functions).\\

We will see that the study of the Berezin transform for any $\beta>-1$ is reduced to the study of it in the case $\beta=\beta_0\in]-1,0]$. So let us concentrate on this particular case.\\
Using a formula page 47 in \cite{Ma-Ob-So}, we have
\begin{equation}\label{eq(2.2)}
\ _2F_1\left(\left.
\begin{array}{cc}
1,&\alpha+\beta+2\\
& \beta+1
\end{array}\right|\xi\right) =\frac{\ _2F_1\left(\left.
\begin{array}{cc}
\beta,&-(\alpha+1)\\
& \beta+1
\end{array}\right|\xi\right)}{(1-\xi)^{\alpha+2}}=:\frac{G_{\alpha,\beta}(\xi)}{(1-\xi)^{\alpha+2}}
\end{equation}
Thus thanks to  \cite{Gh-Sn}, we obtain
$$G_{\alpha,\beta}(\xi)=\sum_{n=0}^{+\infty} \frac{\beta}{n+\beta}\frac{(-\alpha-1)_n}{n!}\xi^n=\sum_{n=0}^{+\infty} \frac{\beta}{n+\beta} {\alpha+1\choose n}(-\xi)^n.$$
Moreover, as we will see, the function $G_{\alpha,\beta}$ plays a fundamental role in the proofs. For this reason, we start here by collecting some of its properties.

\begin{lem}\label{l1}
For every $-1<\alpha,\beta<+\infty$ and $\xi\in\mathbb D$, we have
\begin{enumerate}
  \item\qquad  $\ds \xi G_{\alpha,\beta}'(\xi)=\beta\left((1-\xi)^{\alpha+1}-  G_{\alpha,\beta}(\xi)\right)$
  \item $\begin{array}{lcl}
           G_{\alpha+1,\beta}(\xi)&=&\ds \frac{1}{\alpha+\beta+2}\left(\xi(1-\xi)G_{\alpha,\beta}'(\xi)+(\alpha+\beta+2-\beta\xi)G_{\alpha,\beta}(\xi)\right)\\
           &=&\ds\frac{1}{\alpha+\beta+2}\left((\alpha+2)G_{\alpha,\beta}(\xi)+\beta(1-\xi)^{\alpha+2}\right)
         \end{array}$
  \item \qquad $\ds \xi G_{\alpha,\beta+1}(\xi)=\frac{\beta+1}{\alpha+\beta+2}\left(G_{\alpha,\beta}(\xi)-(1-\xi)^{\alpha+2}\right)$
  \item If $t\in[0,1]$ then we have $\ds (1-t)^{\alpha+2}\leq G_{\alpha,\beta+1}(t)\leq G_{\alpha,\beta}(t)$
  \item If $-1< \beta\leq 0$ then the function $G_{\alpha,\beta}$ is continuous and non-decreasing on $[0,1]$. Thus we obtain
$$\inf_{0\leq t\leq 1} G_{\alpha,\beta}(t)=1\quad \text{and}\quad \sup_{0\leq t\leq 1} G_{\alpha,\beta}(t)=G_{\alpha,\beta}(1)=(\alpha+1)\mathscr B(\alpha+1,\beta+1).$$
\end{enumerate}
\end{lem}
A part of this lemma was cited in \cite{Sn}. For the completeness of the paper, we give the complete proof.
\begin{proof}
Let  $-1<\alpha,\beta<+\infty$ and $\xi\in\mathbb D$.
\begin{itemize}
  \item The formula $(1)$ is simple. Indeed, $$\xi G_{\alpha,\beta}'(\xi)=\ds \beta\sum_{n=0}^{+\infty} \frac{n(-\alpha-1)_n}{n+\beta}\frac{\xi^n}{n!}=\beta\sum_{n=0}^{+\infty} \left(1-\frac{\beta}{n+\beta}\right)(-\alpha-1)_n\frac{\xi^n}{n!}=\beta\left((1-\xi)^{\alpha+1}-  G_{\alpha,\beta}(\xi)\right).$$
  \item To prove the first equality in $(2)$, we compute the term of the right-hand side:
    $$\begin{array}{lcl}
         R(\xi)&:=&\ds\xi(1-\xi)G_{\alpha,\beta}'(\xi)+(\alpha+\beta+2-\beta\xi)G_{\alpha,\beta}(\xi)\\
         &=&\ds \beta\left[\sum_{n=0}^{+\infty} \frac{n+\alpha+\beta+2}{n+\beta}(-\alpha-1)_n\frac{\xi^n}{n!}-\sum_{n=0}^{+\infty} (-\alpha-1)_n\frac{\xi^{n+1}}{n!}\right]\\
         &=&\ds \beta\left[\sum_{n=0}^{+\infty} \frac{n+\alpha+\beta+2}{n+\beta}(-\alpha-1)_n\frac{\xi^n}{n!}-\sum_{m=1}^{+\infty} (-\alpha-1)_{m-1}\frac{\xi^m}{(m-1)!}\right]\\
         &=&\ds \beta\sum_{n=0}^{+\infty} \frac{1}{n+\beta}\Bigl[(n+\alpha+\beta+2)(-\alpha-1)_n- n(n+\beta) (-\alpha-1)_{n-1}\Bigr]\frac{\xi^n}{n!}\\
         &=&\ds \beta\sum_{n=0}^{+\infty} \frac{(-\alpha-1)_{n-1}}{n+\beta}\Bigl[(n+\alpha+\beta+2)(n-\alpha-2)- n(n+\beta) \Bigr]\frac{\xi^n}{n!}\\
         &=&\ds \beta(\alpha+\beta+2)\sum_{n=0}^{+\infty} \frac{(-\alpha-2)_n}{n+\beta}\frac{\xi^n}{n!}\\
         &=&\ds (\alpha+\beta+2)G_{\alpha+1,\beta}(\xi).
      \end{array}
    $$
    While the second equality in $(2)$ is a simple consequence of $(1)$ and the first equality in $(2)$:
    $$\begin{array}{lcl}
           G_{\alpha+1,\beta}(\xi)&=&\ds \frac{1}{\alpha+\beta+2}\left(\xi(1-\xi)G_{\alpha,\beta}'(\xi)+(\alpha+\beta+2-\beta\xi)G_{\alpha,\beta}(\xi)\right)\\
           &\overset{(1)}=&\ds \frac{1}{\alpha+\beta+2}\left[\beta(1-\xi)\left((1-\xi)^{\alpha+1}-G_{\alpha,\beta}(\xi)\right)+(\alpha+\beta+2-\beta\xi) G_{\alpha,\beta}(\xi)\right]\\
           &=&\ds\frac{1}{\alpha+\beta+2}\left((\alpha+2)G_{\alpha,\beta}(\xi)+\beta(1-\xi)^{\alpha+2}\right)
         \end{array}$$
  \item Again to prove Equality $(3)$, we make a change of variable and we use the second equality in $(2)$ as follows:
    $$\begin{array}{lcl}
           \xi G_{\alpha,\beta+1}(\xi)&=&\ds (\beta+1)\sum_{n=0}^{+\infty}\frac{1}{n+\beta+1}(-\alpha-1)_n\frac{\xi^{n+1}}{n!}\qquad\qquad (m=n+1)\\
           &=&\ds (\beta+1)\sum_{m=1}^{+\infty}\frac{1}{m+\beta}(-\alpha-1)_{m-1}\frac{\xi^m}{(m-1)!}\\
           &=&\ds -\frac{\beta+1}{\alpha+2}\sum_{m=0}^{+\infty}\frac{m}{m+\beta}(-\alpha-2)_m\frac{\xi^m}{m!}\\
           &=&\ds -\frac{\beta+1}{\alpha+2}\left[\sum_{m=0}^{+\infty}(-\alpha-2)_m\frac{\xi^m}{m!}-\beta \sum_{m=0}^{+\infty}\frac{1}{m+\beta}(-\alpha-2)_m\frac{\xi^m}{m!}\right]\\
           &=&\ds -\frac{\beta+1}{\alpha+2}\Bigl((1-\xi)^{\alpha+2}-G_{\alpha+1,\beta}(\xi)\Bigr)\\
           &=&\ds\frac{\beta+1}{\alpha+\beta+2}\Bigl(G_{\alpha,\beta}(\xi)-(1-\xi)^{\alpha+2}\Bigr).
         \end{array}$$
  \item Using the Stirling formula, one can see that $G_{\alpha,\beta}$ is extendable continuously to $\overline{\mathbb D}$, hence it is continuous on $[0,1]$. To prove $(4)$ we start by taking $t\in[0,1[$ then using Formula \eqref{eq(2.2)}, we obtain
    $$1\leq \ _2F_1\left(\left.
        \begin{array}{cc}
            1,&\alpha+\beta+3\\
            & \beta+2
        \end{array}\right|t\right) =\frac{G_{\alpha,\beta+1}(t)}{(1-t)^{\alpha+2}}.$$
    It follows that $(1-t)^{\alpha+2}\leq G_{\alpha,\beta+1}(t)$ for every $t\in[0,1]$ and so the first inequality is proved. For the second inequality (of the right-hand side) in $(4)$, we set $h_{\alpha,\beta}(t)=G_{\alpha,\beta}(t)- G_{\alpha,\beta+1}(t)$ and we take $\alpha=\alpha_0+s$ where $-1<\alpha_0\leq0$ and $s\in\mathbb N$ and we proceed by induction on $s$ to prove that $h_{\alpha,\beta}\geq 0$ on $[0,1]$.
    \begin{itemize}
      \item If $\alpha=0$ then $h_{0,\beta}(t)=\frac{t}{(\beta+1)(\beta+2)}\geq0$. So we assume that $\alpha=\alpha_0\in]-1,0[$ (i.e. $s=0$). It follows that for every $t$,
        $$h_{\alpha,\beta}(t)=\sum_{n=1}^{+\infty}\left(\frac{\beta}{n+\beta}-\frac{\beta+1}{n+\beta+1}\right)(-\alpha-1)_n\frac{t^n}{n!} =\sum_{n=1}^{+\infty}\left(\frac{n(\alpha+1)}{(n+\beta)(n+\beta+1)}\right)\frac{\Gamma(n-\alpha-1)}{\Gamma(-\alpha)}\frac{t^n}{n!}\geq0.$$
      \item We assume now that $h_{\alpha,\beta}\geq 0$ for some $\alpha=\alpha_0+s$ and let we prove that $h_{\alpha+1,\beta}\geq 0$. Thanks to $(2)$ and the first inequality in $(4)$,
          $$\begin{array}{l}
               h_{\alpha+1,\beta}(t)=\ds G_{\alpha+1,\beta}(t)- G_{\alpha+1,\beta+1}(t)\\
               =\ds \frac{1}{\alpha+\beta+2}\left((\alpha+2)G_{\alpha,\beta}(t)+\beta(1-t)^{\alpha+2}\right) -\frac{1}{\alpha+\beta+3}\left((\alpha+2)G_{\alpha,\beta+1}(t)+(\beta+1)(1-t)^{\alpha+2}\right)\\
               =\ds \frac{\alpha+2}{\alpha+\beta+2}(G_{\alpha,\beta}(t)-G_{\alpha,\beta+1}(t)) +\frac{\alpha+2}{(\alpha+\beta+2)(\alpha+\beta+3)}(G_{\alpha,\beta+1}(t)-(1-t)^{\alpha+2})\\
               =\ds \frac{\alpha+2}{\alpha+\beta+2}\left(h_{\alpha,\beta}(t)+\frac{1}{\alpha+\beta+3}(G_{\alpha,\beta+1}(t)-(1-t)^{\alpha+2})\right)\geq0.
            \end{array}
          $$
    \end{itemize}
    This finishes the proof of the last inequality in $(4)$.
    \item Again as in the previous formula, we prove $(5)$  by induction on $s$.
        \begin{itemize}
            \item First case: $s=0$: so $\alpha=\alpha_0\in]-1,0]$. Using a formula on page 41 in \cite{Ma-Ob-So}, we obtain for every $t\in[0,1[$,
                $$G_{\alpha,\beta}'(t)=-\frac{\beta(\alpha+1)}{\beta+1}\ _2F_1\left(\left.
                    \begin{array}{cc}
                    \beta+1,&-\alpha\\
                    & \beta+2
                    \end{array}\right|t\right)> 0.$$
                This means that $G_{\alpha,\beta}$ is in fact increasing on $[0,1]$.
            \item Suppose that the function $G_{\alpha,\beta}$ is increasing on $[0,1]$ for some $\alpha>-1$. Thanks to Equality $(2)$, we  have
                $$G_{\alpha+1,\beta}(t)=\frac{1}{\alpha+\beta+2}\left((\alpha+2)G_{\alpha,\beta}(t)+\beta(1-t)^{\alpha+2}\right).$$
                By differentiating both sides, we obtain
                $$G'_{\alpha+1,\beta}(t)=\frac{1}{\alpha+\beta+2}\left((\alpha+2)G'_{\alpha,\beta}(t)-\beta(\alpha+2) (1-t)^{\alpha+1}\right).$$
                It follows that $G_{\alpha+1,\beta}$ is also increasing on $[0,1]$.
        \end{itemize}
        We conclude so that $G_{\alpha,\beta}$ is increasing on $[0,1]$ for every $\alpha>-1$. Hence, for every  $t\in[0,1]$ we have $$1=G_{\alpha,\beta}(0)\leq G_{\alpha,\beta}(t)\leq G_{\alpha,\beta}(1).$$
        To finish the proof we claim that the last equality $G_{\alpha,\beta}(1)=(\alpha+1)\mathscr B(\alpha+1,\beta+1)$  is so simple. Indeed,
        $$\begin{array}{lcl}
            \mathscr B(\alpha+1,\beta+1)&= &\ds \int_0^1t^\beta(1-t)^\alpha dt= \ds \sum_{n=0}^{+\infty}(-1)^n{\alpha \choose n } \int_0^1t^{\beta+n} dt\\
             &=&\ds \sum_{n=0}^{+\infty}\frac{(-1)^n}{n+\beta+1}{\alpha\choose  n } =-\sum_{m=1}^{+\infty}\frac{(-1)^m}{m+\beta}{\alpha\choose  m-1}\\
             &=&-\ds\sum_{m=1}^{+\infty}m\frac{(-1)^m}{m+\beta}{\alpha+1\choose  m}=\frac{1}{\alpha+1}\sum_{m=0}^{+\infty}\frac{\beta(-1)^m}{m+\beta}{\alpha+1\choose  m}\\
             &=&\ds \frac{1}{\alpha+1}G_{\alpha,\beta}(1).
        \end{array}
        $$
\end{itemize}
\end{proof}
\begin{rem}
    It is simple to see that for every $t\in[0,1]$ we have $\lim_{\beta\to+\infty}G_{\alpha,\beta}(t)=(1-t)^{\alpha+1}$. Thus we can replace $(4)$ in Lemma \ref{l1} by  $\ds (1-t)^{\alpha+1}\leq G_{\alpha,\beta+1}(t)\leq G_{\alpha,\beta}(t)$ for every $t\in[0,1]$.
\end{rem}

\section{Berezin transform}
Our aim now is to define the Berezin transform  $\mathbb B_{\alpha,\beta}$ and give some of its properties, and then to prove Theorem \ref{th1}.
\subsection{Definition and properties of  $\mathbb B_{\alpha,\beta}$}
\begin{defn}
For every $-1<\alpha,\beta<+\infty$ and every measurable function $f$ on the unit disk $\mathbb D$,  the Berezin transform $\mathbb B_{\alpha,\beta}$ of $f$ is given by
$$\mathbb B_{\alpha,\beta}f(z)=\int_{\mathbb D}f(w)\frac{|\mathbb K_{\alpha,\beta}(w,z)|^2}{\|\mathbb K_{\alpha,\beta}(.,z)\|^2_{2,\mu_{\alpha,\beta}}} d\mu_{\alpha,\beta}(w)
$$
where $\| .\|_{2,\mu_{\alpha,\beta}}$ is the norm of $L^2(\mathbb D,\mu_{\alpha,\beta})$
\end{defn}
It is well known that $\|\mathbb K_{\alpha,\beta}(.,z)\|^2_{2,\mu_{\alpha,\beta}}=\mathbb K_{\alpha,\beta}(z,z)=\mathcal K_{\alpha,\beta}(|z|^2)$. Thus using \eqref{eq(2.1)}, it is simple to obtain that the Berezin transform is $\mathbb B_{\alpha,\beta}f=\mathbb B_{\alpha,\beta_0}f$.
Thus we assume from now that $\beta=\beta_0\in]-1,0]$. In  this case one can see that   $\mathbb B_{\alpha,\beta}$ is well defined on $L^1(\mathbb D,\mu_{\alpha,\beta})$.\\

\begin{prop}
The Berezin transform $\mathbb B_{\alpha,\beta}$ is one-to-one  on $L^1(\mathbb D,\mu_{\alpha,\beta})$.
\end{prop}
\begin{proof}
Since $\mathbb B_{\alpha,\beta}$ is linear on $L^1(\mathbb D,\mu_{\alpha,\beta})$, then to prove the result it suffices to consider $f\in L^1(\mathbb D,\mu_{\alpha,\beta})$ such that $\mathbb B_{\alpha,\beta}f= 0$ and we prove that $f=0$. As
$$\mathbb B_{\alpha,\beta}f(z)=\frac{1}{\mathcal K_{\alpha,\beta}(|z|^2)}\int_{\mathbb D}f(w)\left|\mathcal K_{\alpha,\beta}(z\overline{w})\right|^2d\mu_{\alpha,\beta}(w)
$$
Then if we set
$$u_{\alpha,\beta}(z)=\mathcal K_{\alpha,\beta}(|z|^2)\mathbb B_{\alpha,\beta}f(z)=\int_{\mathbb D}f(w)\mathcal K_{\alpha,\beta}(\overline{z}w)\mathcal K_{\alpha,\beta}(z\overline{w})d\mu_{\alpha,\beta}(w),
$$
we obtain
$$\frac{\partial^{n+k}u_{\alpha,\beta}}{\partial z^n\partial \overline{z}^k}(0)=\frac{n!k!(\alpha+\beta+2)_n(\alpha+\beta+2)_k}{(\beta+1)_n(\beta+1)_k}\int_{\mathbb D}f(w)\overline{w}^nw^kd\mu_{\alpha,\beta}(w)=0,\quad \forall\; n,k\geq0.
$$
It follows that
$$\int_{\mathbb D}f(w)\overline{w}^nw^kd\mu_{\alpha,\beta}(w)=0,\quad \forall\; n,k\geq0.
$$
The density of polynomials in  $L^1(\mathbb D,\mu_{\alpha,\beta})$ gives that $f=0$, and the injectivity of the Berezin transform is proved.
\end{proof}
If we denote by $\mathcal C(\overline{\mathbb D})$ the space of continuous functions on $\overline{\mathbb D}$ and $\mathcal C_0(\overline{\mathbb D})$ the subspace of continuous functions on $\overline{\mathbb D}$ that vanish on the boundary $\partial \mathbb D$ of $\mathbb D$ then we have the following result:
\begin{prop}
If $f\in\mathcal C(\overline{\mathbb D})$ then $\mathbb B_{\alpha,\beta}f\in\mathcal C(\overline{\mathbb D})$ and $f-\mathbb B_{\alpha,\beta}f\in\mathcal C_0(\overline{\mathbb D})$.
\end{prop}
\begin{proof}
Using  Formula \eqref{eq(2.2)} and the change of variable $w=\varphi_z(\zeta)=\frac{z-\zeta}{1-\zeta \overline{z}}$ (the M\"obius transform of the unit disk), we obtain
$$\begin{array}{lcl}
     \mathbb B_{\alpha,\beta}f(z)&=&\ds \frac{1}{\mathcal K_{\alpha,\beta}(|z|^2)}\int_{\mathbb D}f(w)\left|\mathcal K_{\alpha,\beta}(w\overline{z})\right|^2d\mu_{\alpha,\beta}(w)\\
 &=&\ds \frac{(1-|z|^2)^{\alpha+2}}{G_{\alpha,\beta}(|z|^2)}\int_{\mathbb D}f(w)\frac{\left|G_{\alpha,\beta}(w\overline{z})\right|^2}{|1-w\overline{z}|^{2(\alpha+2)}}d\mu_{\alpha,\beta}(w)\\
&=&\ds \frac{1}{\mathscr B(\alpha+1,\beta+1)}\frac{1}{G_{\alpha,\beta}(|z|^2)}\int_{\mathbb D}f(\varphi_z(\zeta))\left| G_{\alpha,\beta}(\varphi_z(\zeta)\overline{z})\right|^2|\varphi_z(\zeta)|^{2\beta}(1-|\zeta|^2)^\alpha dA(\zeta)\\
  \end{array}
$$

It is well known that for every $\xi\in\partial\mathbb D$, we have $\varphi_z(\zeta)$ tends to $\xi$ as $z\to \xi$ for every $\zeta\in\mathbb D$. Hence using the continuity of $f$ at $\xi$ and the last equality in (5), Lemma \ref{l1},  we obtain
$$\lim_{z\to\xi}\mathbb B_{\alpha,\beta}f(z)=\frac{G_{\alpha,\beta}(1)}{\mathscr B(\alpha+1,\beta+1)}f(\xi)\int_{\mathbb D}(1-|\zeta|^2)^\alpha dA(\zeta)=f(\xi).
$$
\end{proof}

\subsection{Proof of Theorem \ref{th1}}
To prove Theorem \ref{th1}, we need the following lemma:

\begin{lem}\label{l2}
For every  $c>-1,\ d>-\alpha-3$ and $w\in\mathbb D$, we set
$$\mathcal I_{c,d}(w):=\int_{\mathbb D}\frac{\left|\mathcal K_{\alpha,\beta}(z\overline{w})\right|^2}{\mathcal K_{\alpha,\beta}(|z|^2)}|z|^{2c}(1-|z|^2)^ddA(z).
$$
Then
$$\mathcal I_{c,d}(w)\approx \left\{
\begin{array}{lcl}
  1&if& \alpha<d\\
   \ds\log\left(\frac{1}{1-|w|^2}\right)&if& \alpha=d\\
   \ds\frac{1}{(1-|w|^2)^{\alpha-d}}&if& \alpha>d.
\end{array}\right.
$$
as $|w| \to1$.
\end{lem}
Here we set $u\lesssim v$ to say that there exists $c>0$ such that $u\leq cv$ and $u\approx v$  if $u\lesssim v$ and $v\lesssim u$.
\begin{proof}
Using the function $G_{\alpha,\beta}$ defined above, we obtain
$$\ds\mathcal I_{c,d}(w)=\ds \int_{\mathbb D}\frac{\left|\mathcal K_{\alpha,\beta}(z\overline{w})\right|^2}{\mathcal K_{\alpha,\beta}(|z|^2)}|z|^{2c}(1-|z|^2)^ddA(z)
=\ds \int_{\mathbb D}\frac{\left|\mathcal K_{\alpha,\beta}(z\overline{w})\right|^2}{G_{\alpha,\beta}(|z|^2)}|z|^{2c}(1-|z|^2)^{\alpha+2+d}dA(z)
$$
Using Lemma \ref{l1}, we obtain
$$\frac{1}{G_{\alpha,\beta}(1)}\mathcal J_{c,d}(w)\leq \mathcal I_{c,d}(w)\leq \mathcal J_{c,d}(w)$$
where
$$\begin{array}{lcl}
     \mathcal J_{c,d}(w)&=&\ds\int_{\mathbb D}\left|\mathcal K_{\alpha,\beta}(z\overline{w})\right|^2|z|^{2c}(1-|z|^2)^{\alpha+2+d}dA(z)\\
&=&\ds  \mathscr B(c+1,\alpha+d+3)\ _4F_3\left(\left.
\begin{array}{c}
1,\ \alpha+\beta+2,\ \alpha+\beta+2,\ c+1\\
\beta+1,\ \beta+1,\ \alpha+c+d+4
\end{array}\right||w|^2\right).
  \end{array}
$$

Now using the Stirling formula, for $|w|\to1$ we find
$$\mathcal I_{c,d}(w)\approx\mathcal J_{c,d}(w)\approx\sum_{n=0}^{+\infty} (n+1)^{\alpha-d-1}|w|^{2n}
$$
and the result follows from the fact that
$$\sum_{n=0}^{+\infty} (n+1)^{\alpha-d-1}|w|^{2n}\approx \left\{
\begin{array}{lcl}
  1&if& \alpha-d<0\\
   \ds\log\left(\frac{1}{1-|w|^2}\right)&if&\alpha-d=0\\
   \ds\frac{1}{(1-|w|^2)^{\alpha-d}}&if& \alpha-d>0.
\end{array}\right.
$$
\end{proof}
Now we can prove  Theorem \ref{th1} that consists to prove that:\\
\emph{The Berezin transform $\mathbb B_{\alpha,\beta}:L^p(\mathbb D,\mu_{a,b})\longrightarrow L^p(\mathbb D,\mu_{a,b})$ is bounded if and only if
$$p(\alpha+1)>a+1\quad and\quad \left\{
\begin{array}{lcl}
   b\leq \beta& if& p=1\\
   b< p(\beta+1)-1& if& p>1.
\end{array}
\right.$$}
\begin{proof} The proof will be done in two steps:
\begin{itemize}
  \item We assume first that the Berezin transform $\mathbb B_{\alpha,\beta}:L^p(\mathbb D,\mu_{a,b})\longrightarrow L^p(\mathbb D,\mu_{a,b})$ is bounded. Let $q$ be the conjugate exponent of $p$; i.e. $\frac{1}{p}+\frac{1}{q}=1$ ($q=\infty$ if $p=1$). we deduce so that adjoint operator $\mathbb B_{\alpha,\beta}^*$ of $\mathbb B_{\alpha,\beta}$ with respect to the dual action induced by the inner product $\langle .,.\rangle_{\mu_{a,b}}$ maps continuously $L^q(\mathbb D,\mu_{a,b})$ into itself.\\
      It is simple to see that $\mathbb B_{\alpha,\beta}^*$ is given by
      $$\mathbb B_{\alpha,\beta}^*g(w)=|w|^{2(\beta-b)}(1-|w|^2)^{\alpha-a}\int_{\mathbb D}g(z)\frac{\left|\mathcal K_{\alpha,\beta}(z\overline{w})\right|^2}{\mathcal K_{\alpha,\beta}(|z|^2)}d\mu_{a,b}(z).
$$
\begin{itemize}
  \item First case: If $p>1$ then by taking $g_N(z)=(1-|z|^2)^N$ for $N$ large enough and applying Lemma \ref{l2}, we obtain
$$\mathbb B_{\alpha,\beta}^*g_N(w)\approx|w|^{2(\beta-b)}(1-|w|^2)^{\alpha-a}\mathcal I_{b,a+N}(w)\approx|w|^{2(\beta-b)}(1-|w|^2)^{\alpha-a}.$$

Thus $\mathbb B_{\alpha,\beta}^*g_N\in L^q(\mathbb D,\mu_{a,b})$ gives that $p(\beta+1)>1+b$ and $p(\alpha+1)>1+a$.
  \item Second case: If $p=1$ then applying $\mathbb B_{\alpha,\beta}^*$ to the constant function $g\equiv 1$ and using again Lemma \ref{l2}, we obtain $\alpha>a$ and $\beta\geq b$.
\end{itemize}
  \item Conversely, assume that
    $$p(\alpha+1)>a+1\quad and\quad \left\{
        \begin{array}{lcl}
            b\leq \beta& if& p=1\\
            b< p(\beta+1)-1& if& p>1.
        \end{array}\right.$$
     Let we prove that  $\mathbb B_{\alpha,\beta}$ is bounded on $L^p(\mathbb D,\mu_{a,b})$. If $p=1$ then Lemma \ref{l2} gives that $\mathbb B_{\alpha,\beta}$ is bounded on $L^p(\mathbb D,\mu_{a,b})$ with norm
     $$\|\mathbb B_{\alpha,\beta}\|_{\mathscr L(L^1(\mathbb D,\mu_{a,b}))}\leq \sup_{w\in\mathbb D}\int_{\mathbb D}\omega(z,w)d\mu_{a,b}(z)$$
     where
     $$\omega(z,w)=\frac{\mathscr B(a+1,b+1)}{\mathscr B(\alpha+1,\beta+1)}|w|^{2(\beta-b)}(1-|w|^2)^{\alpha-a}\frac{\left|\mathcal K_{\alpha,\beta}(z\overline{w})\right|^2} {\mathcal K_{\alpha,\beta}(|z|^2)}.$$
     For $p>1$ (and $q$ its conjugate), we  consider the function $h_{s,\tau}$ defined by  $$h_{s,\tau}(z)=\frac{1}{|z|^{2s}(1-|z|^2)^\tau}.$$
     If the parameters $s$ and $\tau$ satisfy the following conditions
     $$\frac{b-\beta}{p}<s<\frac{b+1}{p}\quad\text{and}\quad\frac{a-\alpha}{p}<\tau<\frac{a+\alpha+3}{p},$$
      then thanks to Lemma \ref{l2},
     $$\int_{\mathbb D}h_{s,\tau}(z)^p\omega(z,w)d\mu_{a,b}(z)\approx|w|^{2(\beta-b)}(1-|w|^2)^{\alpha-a}\mathcal I_{b-sp,a-\tau p}(w) \approx\frac{|w|^{2(\beta-b)}(1-|w|^2)^{\alpha-a}}{(1-|w|^2)^{\alpha-a+tq}}.
     $$
     It follows that $$\int_{\mathbb D}h_{s,\tau}(z)^p\omega(z,w)d\mu_{a,b}(z)\lesssim\frac{|w|^{2(\beta-b)}(1-|w|^2)^{a-\alpha}}{(1-|w|^2)^{\alpha-(a-\tau p}} \lesssim h_{s,\tau}(w)^p.
     $$
     Now, if $$0<s<\frac{\beta+1}{q}\quad\text{and}\quad -\frac{\alpha+2}{q}<\tau<\frac{\alpha+1}{q},$$ then thanks to Lemma \ref{l1} and using the same techniques as in the proof of Lemma \ref{l2},  we obtain
     $$\begin{array}{lcl}
         \ds\int_{\mathbb D}h_{s,\tau}(w)^q\omega(z,w)d\mu_{a,b}(w)&\approx&\ds \int_{\mathbb D}|w|^{2(\beta-sq)}(1-|w|^2)^{\alpha-\tau p}\frac{\left|\mathcal K_{\alpha,\beta}(z\overline{w})\right|^2}{\mathcal K_{\alpha,\beta}(|z|^2)}d\mu_{a,b}(w)\\
&\approx&\ds (1-|z|^2)^{\alpha+2}\sum_{n=0}^{+\infty}\frac{(\alpha+\beta+2)_n(\beta-qs+1)_n}{(\beta+1)_n(\alpha+\beta+2-qs-q\tau)_n}|z|^{2n}\\
&\approx&\ds\frac{1}{(1-|z|^2)^{\tau q}}\lesssim h_{s,\tau}(z)^q
       \end{array}
     $$
     Using the assumptions, one can see that
     $$\left]\frac{b-\beta}{p},\frac{b+1}{p}\right[\cap\left]0,\frac{\beta+1}{q}\right[\neq\emptyset\quad \text{and}\quad \left]\frac{a-\alpha}{p},\frac{a+\alpha+3}{p}\right[\cap\left]-\frac{\alpha+2}{q},\frac{\alpha+1}{q}\right[\neq\emptyset.
     $$
This shows the existence of $s,\tau$ such that the function $h_{s,\tau}$ satisfies Schur's test. Hence we conclude the continuity of the Berezin transform and so the proof of Theorem \ref{th1} is achieved.
\end{itemize}
\end{proof}
\begin{rem}
    If $u\in L^2(\mathbb D,d\mu_{\alpha,\beta})$ is a harmonic function on $\mathbb D$ then $\mathbb B_{\alpha,\beta} u=u$.  This result is not true in general for harmonic functions on $\mathbb D^*$.
\end{rem}
Indeed, the first statement can be deduced from the fact that any harmonic function $u$ on $\mathbb D$ can be written as  $u=f+\overline{g}$ where $f$ and $g$ are holomorphic on $\mathbb D$. While for the second statement, if we take $v(z)=\log|z|^2$ then by a simple computation, one can show that $$\mathbb B_{0,\beta} v(z)= \frac{1-|z|^2
}{\beta+1-\beta|z|^2}.$$
Thus $\mathbb B_{0,\beta} v\neq v$.
\section{Bergman-Poincar\'e metric and application}
In what follows, we recall the Bergman-Poincar\'e metric introduced by the second author in \cite{Sn} and we define the notion of BMO (bounded mean oscillation) functions, then we prove Theorem \ref{th2} that gives the uniform continuity of the Berezin transform of any element of the modified Bergman space with respect to the Bergman-Poincar\'e metric (that improves a result given in \cite{BBCZ}).
\subsection{BMO in Bergman-Poincar\'e metric}
To start, let we consider $\alpha>-1,\ -1<\beta\leq 0$ as above. Then it is proved in \cite{Sn} that the function
$$\kappa_{\alpha,\beta}(z)=\log(\mathbb K_{\alpha,\beta}(z,z))=\log(\mathcal K_{\alpha,\beta}(|z|^2))$$ is a $\mathcal C^\infty$ subharmonic function on $\mathbb D$. Thus if we set $\varrho_{\alpha,\beta}^2(z)=\frac{\partial^2 \kappa_{\alpha,\beta}}{\partial z\partial \overline{z}}(z)$ then we obtain
$$\varrho^2_{\alpha,\beta}(z)=\frac{\mathcal K'_{\alpha,\beta}(|z|^2)}{\mathcal K_{\alpha,\beta}(|z|^2)}+|z|^2\frac{\mathcal K''_{\alpha,\beta}(|z|^2)}{\mathcal K_{\alpha,\beta}(|z|^2)}-|z|^2\left(\frac{\mathcal K'_{\alpha,\beta}(|z|^2)}{\mathcal K_{\alpha,\beta}(|z|^2)}\right)^2\geq0.
$$
It follows that one can define a metric (called modified Bergman-Poincar\'e metric) $\mathbf{d}_{\alpha,\beta}$ on $\mathbb D$ by setting $\mathbf{d}_{\alpha,\beta}(z,w):=\inf\ell_{\alpha,\beta}(\gamma)$ for every $z,w\in\mathbb D$ where the infimum is taken over all piecewise continuously differentiable curves $\gamma:[0,1]\longrightarrow \mathbb D$ with $\gamma(0)=z,\ \gamma(1)=w$ and $\ell_{\alpha,\beta}(\gamma) $ is the length of $\gamma$ given by $$\ell_{\alpha,\beta}(\gamma):=\int_0^1\varrho_{\alpha,\beta}(\gamma(s))|\gamma'(s)|ds.
$$
It is showed in \cite{Sn} that the space $(\mathbb D,\mathbf{d}_{\alpha,\beta})$ is a complete metric space.\\

For every $f\in L^2(\mathbb D,d\mu_{\alpha,\beta})$, we define the mean oscillation of $f$ as
$$MO(f)(z)=\left(\mathbb B_{\alpha,\beta}(|f|^2)(z)-|\mathbb B_{\alpha,\beta}|f|(z)|^2\right)^{\frac{1}{2}}.$$
It is easy to see that
$$(MO(f)(z))^2=\frac{1}{2\mathcal K^2_{\alpha,\beta}(|z|^2)}\int_{\mathbb D}\int_{\mathbb D}|f(u)-f(v)|^2|\mathcal K_{\alpha,\beta}(z\overline{u})|^2|\mathcal K_{\alpha,\beta}(z\overline{v})|^2d\mu_{\alpha,\beta}(u)d\mu_{\alpha,\beta}(v).$$
We deduce so that the function  $MO(f)$ is well defined on $\mathbb D$.
\begin{defn}
    We say that $f\in BMO$ if $\|f\|_{BMO}<\infty$ where $$\|f\|_{BMO}=\sup\{MO(f)(z);\ z\in\mathbb D\}.$$
\end{defn}
\subsection{Proof of Theorem \ref{th2}}
To simplify the notation, for every $\xi\in\mathbb D$, we set $$\phi_\xi(z)=\frac{\mathcal K_{\alpha,\beta}(z\overline{\xi})}{\sqrt{\mathcal K_{\alpha,\beta}(|\xi|^2)}}$$
 and $ P_\xi$ the orthogonal projection from $\mathcal A^2(\mathbb D,\mu_{\alpha,\beta})$ onto the one-dimensional subspace spanned by $\phi_\xi$. It follows that for every $f\in \mathcal A^2(\mathbb D,\mu_{\alpha,\beta})$,
 \begin{equation}\label{eq(4.1)}
      P_\xi f(z)=\phi_\xi(z)\langle f,\phi_\xi\rangle_{\alpha,\beta}=\frac{f(\xi)}{\mathcal K_{\alpha,\beta}(|\xi|^2)}\mathcal K_{\alpha,\beta}(z\overline{\xi}).
 \end{equation}

 \begin{prop}
     Let $\gamma(t)$ be a smooth curve in $\mathbb D$. Then $$\left\|(I-P_{\gamma(t)})\left(\frac{d}{dt}\phi_{\gamma(t)}\right)\right\|_{2,\mu_{\alpha,\beta}}=|\gamma'(t)|\varrho_{\alpha,\beta}(\gamma(t)).$$
 \end{prop}
 \begin{proof}
     Since $$\phi_{\gamma(t)}(z)=\frac{\mathcal K_{\alpha,\beta}(z\overline{\gamma(t)})}{\sqrt{\mathcal K_{\alpha,\beta}(|\gamma(t)|^2)}}$$ then
     \begin{equation}\label{eq(4.2)}
         \frac{d}{dt}\phi_{\gamma(t)}(z)=\frac{\mathcal K'_{\alpha,\beta}(z\overline{\gamma(t)})}{\sqrt{\mathcal K_{\alpha,\beta}(|\gamma(t)|^2)}} z\overline{\gamma'(t)}-\frac{1}{2} \frac{\mathcal K_{\alpha,\beta}(z\overline{\gamma(t)})\mathcal K'_{\alpha,\beta}(|\gamma(t)|^2)}{\mathcal K^{\frac{3}{2}}_{\alpha,\beta}(|\gamma(t)|^2)}(\gamma'(t)\overline{\gamma(t)}+\gamma(t)\overline{\gamma'(t)}).
     \end{equation}

     By \eqref{eq(4.1)}-\eqref{eq(4.2)}, we obtain
     \begin{equation}\label{eq(4.3)}
         \begin{array}{lcl}
         \ds P_{\gamma(t)}\left(\frac{d}{dt}\phi_{\gamma(t)}\right)(z)&=&\ds \frac{d}{dt}\phi_{\gamma(t)}(\gamma(t))\frac{\mathcal K_{\alpha,\beta}(z\overline{\gamma(t)})}{\mathcal K_{\alpha,\beta}(|\gamma(t)|^2)}\\
         &=&\ds \frac{\mathcal K_{\alpha,\beta}(z\overline{\gamma(t)})\mathcal K'_{\alpha,\beta}(|\gamma(t)|^2)}{2\mathcal K^{\frac{3}{2}}_{\alpha,\beta}(|\gamma(t)|^2)}(\gamma(t)\overline{\gamma'(t)}-\gamma'(t)\overline{\gamma(t)}).
         \end{array}
     \end{equation}
     Combine \eqref{eq(4.2)} with  \eqref{eq(4.3)} to find
     $$\begin{array}{lcl}
     \ds(I-P_{\gamma(t)})\left(\frac{d}{dt}\phi_{\gamma(t)}\right)(z)&=&\ds \frac{\mathcal K'_{\alpha,\beta}(z\overline{\gamma(t)})}{\sqrt{\mathcal K_{\alpha,\beta}(|\gamma(t)|^2)}} z\overline{\gamma'(t)}- \frac{\mathcal K_{\alpha,\beta}(z\overline{\gamma(t)})\mathcal K'_{\alpha,\beta}(|\gamma(t)|^2)}{\mathcal K^{\frac{3}{2}}_{\alpha,\beta}(|\gamma(t)|^2)}\gamma(t)\overline{\gamma'(t)}\\
     &=:&\mathscr A_t(z)-\mathscr D_t(z).
     \end{array}
     $$
     Hence,
     \begin{equation}\label{eq(4.4)}
         \left\|(I-P_{\gamma(t)})\left(\frac{d}{dt}\phi_{\gamma(t)}\right)\right\|^2_{2,\mu_{\alpha,\beta}}=\|\mathscr A_t\|^2_{2,\mu_{\alpha,\beta}}+\|\mathscr D_t\|^2_{2,\mu_{\alpha,\beta}}-2\Re\langle \mathscr A_t,\mathscr D_t\rangle_{\mu_{\alpha,\beta}}.
     \end{equation}
     It is not hard to see that
     \begin{equation}\label{eq(4.5)}
         \|\mathscr D_t\|^2_{2,\mu_{\alpha,\beta}}=\langle \mathscr A_t,\mathscr D_t\rangle_{\mu_{\alpha,\beta}}=\left(\frac{\mathcal K'_{\alpha,\beta}(|\gamma(t)|^2)}{\mathcal K_{\alpha,\beta}(|\gamma(t)|^2)}\right)^2|\gamma(t)|^2|\overline{\gamma'(t)}|^2
     \end{equation}
      where we use the fact that the function $z\longmapsto z\mathcal K'_{\alpha,\beta}(z\overline{\gamma(t)})$ is in $\mathcal A^2(\mathbb D,\mu_{\alpha,\beta})$.\\
     Now using the definition and the orthogonality, we obtain
     \begin{equation}\label{eq(4.6)}
         \begin{array}{lcl}
         \ds\|\mathscr A_t\|^2_{2,\mu_{\alpha,\beta}} &=&\ds \frac{|\gamma'(t)|^2}{\mathcal K_{\alpha,\beta}(|\gamma(t)|^2)}\int_{\mathbb D}\left|z\mathcal K'_{\alpha,\beta}(z\overline{\gamma(t)})\right|^2d\mu_{\alpha,\beta}(z)\\&=&\ds \frac{|\gamma'(t)|^2}{\mathcal K_{\alpha,\beta}(|\gamma(t)|^2)}\int_{\mathbb D}\left(\sum_{n,m=1}^{+\infty}nm\frac{(\alpha+\beta+2)_n (\alpha+\beta+2)_m}{(\beta+1)_n (\beta+1)_m}(\gamma(t))^{n-1}(\overline{\gamma(t)})^{m-1}z^n\overline{z}^m\right) d\mu_{\alpha,\beta}(z)\\
         &=&\ds \frac{|\gamma'(t)|^2}{\mathcal K_{\alpha,\beta}(|\gamma(t)|^2)} \sum_{n=1}^{+\infty} \left(\frac{n(\alpha+\beta+2)_n }{(\beta+1)_n }\right)^2|\gamma(t)|^{2n-2}\int_{\mathbb D}|z|^{2n} d\mu_{\alpha,\beta}(z)\\
         &=&\ds \frac{|\gamma'(t)|^2}{\mathcal K_{\alpha,\beta}(|\gamma(t)|^2)} \sum_{n=1}^{+\infty} n^2\frac{(\alpha+\beta+2)_n }{(\beta+1)_n }|\gamma(t)|^{2n-2}\\
         &=&\ds \frac{|\gamma'(t)|^2}{\mathcal K_{\alpha,\beta}(|\gamma(t)|^2)}\frac{\alpha+\beta+2}{\beta+1} \sum_{k=0}^{+\infty} (k+1)^2\frac{(\alpha+\beta+3)_k }{(\beta+2)_k }|\gamma(t)|^{2k}\\
         &=&\ds \frac{|\gamma'(t)|^2}{\mathcal K_{\alpha,\beta}(|\gamma(t)|^2)}\frac{\alpha+\beta+2}{\beta+1} \ _3F_2\left(\left.
                \begin{array}{c}
                    2,\ 2,\ \alpha+\beta+3\\
                    1,\ \beta+2
                \end{array}\right||\gamma(t)|^2\right).
         \end{array}
     \end{equation}
     Combining  \eqref{eq(4.4)}, \eqref{eq(4.5)} and \eqref{eq(4.6)} we conclude that
     $$\begin{array}{l}
     \ds \left\|(I-P_{\gamma(t)})\left(\frac{d}{dt}\phi_{\gamma(t)}\right)\right\|^2_{2,\mu_{\alpha,\beta}}=\ds \|\mathscr A_t\|^2_{2,\mu_{\alpha,\beta}}-\|\mathscr D_t\|^2_{2,\mu_{\alpha,\beta}}\\
     =\ds |\gamma'(t)|^2\left[\frac{\alpha+\beta+2}{(\beta+1)\mathcal K_{\alpha,\beta}(|\gamma(t)|^2)}\ _3F_2\left(\left.
                \begin{array}{c}
                    2,\ 2,\ \alpha+\beta+3\\
                    1,\ \beta+2
                \end{array}\right||\gamma(t)|^2\right)
        -\left(\frac{\mathcal K'_{\alpha,\beta}(|\gamma(t)|^2)}{\mathcal K_{\alpha,\beta}(|\gamma(t)|^2)}\right)^2|\gamma(t)|^2\right]\\
      =\ds |\gamma'(t)|^2\left[ \frac{1}{\mathcal K_{\alpha,\beta}(|\gamma(t)|^2)}\left(\mathcal K'_{\alpha,\beta}(|\gamma(t)|^2)+|\gamma(t)|^2\mathcal K''_{\alpha,\beta}(|\gamma(t)|^2)\right) -\left(\frac{\mathcal K'_{\alpha,\beta}(|\gamma(t)|^2)}{\mathcal K_{\alpha,\beta}(|\gamma(t)|^2)}\right)^2|\gamma(t)|^2\right]\\
      =\ds |\gamma'(t)|^2\varrho^2_{\alpha,\beta}(|\gamma(t)|^2).
     \end{array}
     $$

 \end{proof}
 \begin{lem} \label{l3}
     Let $f\in BMO$ and $\gamma(t)$ be a smooth curve in $\mathbb D$. Then
     $$\left|\frac{d}{dt}\mathbb B_{\alpha,\beta}f(\gamma(t))\right|\leq 2MO(f)(\gamma(t))|\gamma'(t)|\varrho_{\alpha,\beta}(\gamma(t)).$$
 \end{lem}
 \begin{proof}
     Using the same argument as in the proof of \cite[Th. 2.22]{He-Ko-Zh} modulo some simple modifications, one can deduce that
     $$\frac{d}{dt}\mathbb B_{\alpha,\beta}f(\gamma(t))=2\int_{\mathbb D}\left(f(w)-\mathbb B_{\alpha,\beta}f(\gamma(t))\right)\Re \left((I-P_{\gamma(t)})\left(\frac{d}{dt}\phi_{\gamma(t)}\right)(w)\overline{\phi_{\gamma(t)}(w)}\right)d\mu_{\alpha,\beta}(w).$$
     Hence, thanks to the Cauchy-Schwarz inequality,
     $$\begin{array}{lcl}
        \ds\left|\frac{d}{dt}\mathbb B_{\alpha,\beta}f(\gamma(t))\right|&\leq&
        \ds 2\int_{\mathbb D}|f(w)-\mathbb B_{\alpha,\beta}f(\gamma(t))|\ |\phi_{\gamma(t)}(w)|\left|(I-P_{\gamma(t)})\left(\frac{d}{dt}\phi_{\gamma(t)}\right)(w)\right|d\mu_{\alpha,\beta}(w)\\
        &\leq&\ds 2\Bigl\|(f-\mathbb B_{\alpha,\beta}f(\gamma(t)))\phi_{\gamma(t)}\Bigr\|_{2,\mu_{\alpha,\beta}}\left\|(I-P_{\gamma(t)})\left(\frac{d}{dt}\phi_{\gamma(t)}\right)\right\|_{2,\mu_{\alpha,\beta}}
     \end{array}$$
     To conclude the result, it suffices to apply the previous proposition and use the fact that
     $$\Bigl\|(f-\mathbb B_{\alpha,\beta}f(\gamma(t)))\phi_{\gamma(t)}\Bigr\|_{2,\mu_{\alpha,\beta}}=MO(f)(\gamma(t)).$$
 \end{proof}
Now we can prove Theorem \ref{th2} by proving that if $f\in BMO$, then  we have
    $$|\mathbb B_{\alpha,\beta}f(z)-\mathbb B_{\alpha,\beta}f(w)|\leq 2\|f\|_{BMO}\mathbf{d}_{\alpha,\beta}(z,w)
    $$
    for all $z,w\in\mathbb D$.
    \begin{proof}
        Let $f\in BMO$ and $z,w\in\mathbb D$. Let $\gamma:[0,1]\longrightarrow \mathbb D$ be a smooth curve with $\gamma(0)=z$ and $\gamma(1)=w$. By Lemma \ref{l3},
        $$\begin{array}{lcl}
        \ds |\mathbb B_{\alpha,\beta}f(z)-\mathbb B_{\alpha,\beta}f(w)|&=&\ds \left|\int_0^1 \frac{d}{dt}\mathbb B_{\alpha,\beta}f(\gamma(t))dt\right|\leq\ds 2\int_0^1MO(f)(\gamma(t))|\gamma'(t)|\varrho_{\alpha,\beta}(\gamma(t))\\
        &\leq&\ds 2\|f\|_{BMO}\int_0^1|\gamma'(t)|\varrho_{\alpha,\beta}(\gamma(t))=2\|f\|_{BMO}\ell_{\alpha,\beta}(\gamma)
        \end{array}$$
        The result follows by taking the infimum on all smooth curves $\gamma$.
    \end{proof}

\end{document}